\documentclass[11pt]{amsart}

\usepackage[colorlinks=true, pdfstartview=FitV, linkcolor=blue, citecolor=blue, urlcolor=blue]{hyperref}

\usepackage{amsfonts}
\usepackage{amscd}
\usepackage{amssymb}
\usepackage{amsthm}
\usepackage{amsmath}
\usepackage{comment}
\usepackage{epsfig}
\usepackage[all]{xy}
\usepackage[pdftex]{color}

\theoremstyle{plain}

\theoremstyle{definition}

\theoremstyle{example}

\numberwithin{equation}{section}

\newtheorem{Theorem}{Theorem}[section]
\newtheorem{Proposition}[Theorem]{Proposition} 
\newtheorem{Lemma}[Theorem]{Lemma}

\newtheorem{Definition}[Theorem]{Definition}
\theoremstyle{definition}
\newtheorem{Comment}{Comment}
\newtheorem{Example}[Theorem]{Example}

\newcommand{\arxiv}[1]{\href{http://arxiv.org/abs/#1}{\tt arXiv:\nolinkurl{#1}}}

\def\asl{\widehat{\mathfrak{sl}}}

\newcommand{\bz}{\Bbb{Z}}
\newcommand{\bc}{\mathbb{C}}
\newcommand{\bq}{\Bbb{Q}}

\newcommand{\Fock}{{\bf F}}
\newcommand{\val}{\text{val}}

\newcommand{\ev}{\text{ev}}
\newcommand{\se}{\text{ev}}

\newcommand{\Id}{\text{Id}}

\def\cA{\mathcal{A}}

\def\cR{\mathcal{R}}

\def\KK{\mathbb{K}}

\def\ZZ{\mathbb{Z}}

\def\fh{\mathfrak{h}}

\def\fgl{\mathfrak{gl}}
\def\fsl{\mathfrak{sl}}

\def\dim{\mathrm{dim}}

\makeatletter
\renewcommand{\@makefnmark}{\mbox{\textsuperscript{}}}

\makeatother

\title{Universal Verma modules and the Misra-Miwa Fock space}
\author{Arun Ram}
\email{aram@unimelb.edu.au}

\address{
Department of  Mathematics and Statistics \\
University of Melbourne \\
Parkville VIC 3010 Australia}

\address{
Department of Mathematics\\
University of Wisconsin\\
Madison, WI 53706 }

\author{Peter Tingley}
\email{ptingley@math.mit.edu}
\address{
MIT Department of  Mathematics \\
77 Massachusetts Ave \\
Cambridge, MA, USA 02139 }
\date{Feb 2, 2010}

\begin{document}

\begin{abstract}
The Misra-Miwa $v$-deformed Fock space is a representation of the quantized affine algebra 
$U_v(\asl_\ell)$. It has a standard basis indexed by partitions and the non-zero matrix entries of the 
action of the Chevalley generators with respect to this basis are powers of $v$. Partitions also index the polynomial Weyl modules for $U_q(\fgl_N)$ as $N$ tends to infinity. We explain how the powers of $v$ which appear in the Misra-Miwa Fock space also appear naturally in the context of Weyl modules. The main tool we use is the Shapovalov determinant for a universal Verma module.

\end{abstract}

\maketitle



\footnote{AMS Subject Classifications: Primary 17B37; Secondary 20G42.}

\section{Introduction}

Fock space is an infinite dimensional vector space which is a representation of several important 
algebras, as described in, for example, \cite[Chapter 14]{Kac:1990}.  Here we consider the charge zero part of Fock space, which we denote by $\Fock$, and its $v$-deformation $\Fock_v$. The space $\Fock$ has a standard $\bq$-basis $\{\, |\lambda\rangle\ |\ \hbox{$\lambda$ is a partition}\}$ and $\Fock_v := \Fock \otimes_\bq \bq(v)$. Following Hayashi \cite{Hayashi:1990}, Misra and Miwa \cite{MM:1990} define an action of the quantized universal enveloping algebra $U_v(\asl_\ell)$ on $\Fock_v$. The only non-zero matrix elements $\langle \mu| F_{\bar i} | \lambda \rangle$ of the Chevalley generators $F_{\bar i}$ in terms of the standard basis occur when $\mu$ is obtained by adding a single $\bar{i}$-colored box to $\lambda$, and these are powers of $v$.

We show that these powers of $v$ also appear naturally in the following context: Partitions with at most $N$ parts index polynomial Weyl modules $\Delta(\lambda)$ for the integral quantum group $U_q^\cA(\fgl_N)$. Let $V$ be the standard $N$ dimensional representation of  $U_q^\cA(\fgl_N)$. If the matrix element
$\langle \mu| F_{\bar i} | \lambda \rangle$ is non-zero then, for sufficiently large $N$, $\left( \Delta^\cA(\lambda) \otimes_\cA V\right)  \otimes_\cA \bq(q)$ contains a highest weight vector of weight $\mu$. There is a unique such highest weight vector $v_\mu$ which satisfies a certain triangularity condition with respect to an integral basis of $\Delta^\cA(\lambda) \otimes_\cA V$. We show that the matrix element $\langle \mu| F_{\bar i} | \lambda \rangle$ is equal to $v^{\val_{\phi_{2\ell}}(v_\mu, v_\mu)}$, 
where  $(\cdot, \cdot)$ is the Shapovalov form and $\val_{\phi_{2\ell}}$ is 
the valuation at the cyclotomic polynomial $\phi_{2\ell}$. 

Our proof is computational, making use of the Shapovalov determinant \cite{Shapovalov:1972, KD:1991, KL:1997}. This is a formula for the determinant of the Shapovalov form on 
a weight space of a Verma module.  The necessary computation is most easily done 
in terms of the universal Verma modules introduced in the classical case by Kashiwara \cite{Kashiwara:1985} and studied in the quantum case by Kamita \cite{Kamita}. The statement for Weyl modules is then a straightforward consequence. 

Before beginning, let us discuss some related work. 
In \cite{Kleshchev:1995}, Kleshchev carefully analyzed the $\fgl_{N-1}$ highest weight 
vectors in a Weyl module for $\fgl_N$, and used this information to give modular branching rules for symmetric group representations.    
Brundan and Kleshchev \cite{BK2} have explained that highest weight vectors in 
the restriction of a Weyl module to $\fgl_{N-1}$ give information about highest weight vectors
in a tensor product $\Delta(\lambda)\otimes V$ of a Weyl module with the standard
$N$-dimensional representation of $\fgl_N$.   Our computations 
put a new twist on the analysis of the highest weight vectors 
in $\Delta(\lambda)\otimes V$, as we study them
in their  ``universal'' versions and by the use of the Shapovalov determinant. Our techniques can be viewed as an application of the theory of 
Jantzen \cite{Jantzen:1974} as extended to the quantum case by Wiesner \cite{Wiesner}.

Brundan \cite{Brundan:1998} generalized
Kleshchev's \cite{Kleshchev:1995} techniques and used this information to give modular branching rules for Hecke algebras.  As discussed in \cite{Ariki:1996, LLT:1996}, these branching rules are reflected in the fundamental representation of 
 $\asl_p$ and its crystal graph, recovering much of the structure of the Misra-Miwa Fock space. Using Hecke algebras at a root of unity, Ryom-Hansen \cite{Ryom-Hansen:2004} recovered the full $U_v(\asl_\ell)$ action on Fock space. To complete the picture one should construct a graded category, where multiplication by $v$ in the $\asl_\ell$ representation corresponds to a grading shift. 
Recent work of Brundan-Kleshchev \cite{BrKl:2009} and  Ariki \cite{Ariki:2009} explains that one solution to this problem is through the representation theory of Khovanov-Lauda-Rouquier algebras \cite{KhLa:2009, Rouquier:2009}. It would be interesting to explicitly describe the relationship between their category and the present work. Another related construction due to Brundan-Stroppel considers the case when the Fock space is replaced by $\wedge^m V\otimes \wedge^n V$, where $V$ is the natural $\mathfrak{gl}_\infty$ module and $m, n$ are fixed natural numbers.

We would also like to mention very recent work of  Peng Shan \cite{Shan:2010} which independently develops a similar story to the one presented here, but  using representations of a quantum Schur algebra where we use representations of $U_\varepsilon(\fgl_N)$. The approach taken there is somewhat different, and in particular relies on localization techniques of Beilinson and Bernstein \cite{BB:1993}. 

This paper is arranged as follows. Sections \ref{qua-group} and \ref{pFock} are background on the quantum group $U_q(\fgl_N)$ and the Fock space $\Fock_v$. Sections \ref{UV-section} and \ref{Sform} explain universal Verma modules and the Shapovalov determinant. Section \ref{makeops} contains the statement and proof of our main result relating Fock space and Weyl modules. 

\subsection{Acknowledgments}

We thank M. Kashiwara, A. Kleshchev, T. Tanisaki, R. Virk and B. Webster for helpful discussions. 
The first author was partly supported by NSF Grant DMS-0353038 and Australian Research Council Grants DP0986774 and DP0879951. The second author was partly supported by the Australia Research Council grant DP0879951 and NSF grant DMS-0902649.

\section{The quantum group $U_q(\fgl_N)$ and its integral form $U_q^\cA(\fgl_N)$} \label{qua-group}

This is a very brief review, intended mainly to fix notation. With slight modifications the construction in this section works in the generality of symmetrizable Kac-Moody algebras. See \cite[Chapters 6 and 9]{CP} for details. 

\subsection{The rational quantum group
} \label{the_q_groups}
$U_q(\fgl_N)$ is the associative algebra over the field of rational 
functions $\bq(q)$ generated by 
\begin{equation}
X_1,\ldots, X_{N-1}, \quad
Y_1,\ldots, Y_{N-1},\quad\hbox{and}\quad
L_1^{\pm1},\ldots, L_N^{\pm1},
\end{equation}
with relations
$$L_iL_j=L_jL_i, \quad 
L_iL_i^{-1}= L_i^{-1}L_i =1,
\qquad
X_i Y_j - Y_j X_i = \delta_{i,j} \frac{L_iL_{i+1}^{-1} - L_{i+1} L_i^{-1}}{q-q^{-1}},
$$ 
\begin{equation} L_i X_j L_i^{-1}=
\begin{cases} qX_j, \quad \text{ if } i=j, \\
q^{-1}X_j, \quad \text{ if } i=j+1, \\
X_j \quad \text{ otherwise};
\end{cases}
\qquad\qquad
L_i Y_j L_i^{-1}=
\begin{cases} q^{-1}Y_j, \quad \text{ if } i=j, \\
q Y_j, \quad \text{ if } i=j+1, \\
Y_j, \quad \text{ otherwise};
\end{cases}
\end{equation}
$$X_iX_j = X_j X_i\quad\hbox{and}\quad Y_i Y_j=Y_jY_i,
\qquad\hbox{if $|i-j| \geq 2$,}$$
$$
X_i^{2} X_j -(q+q^{-1}) X_iX_jX_i + X_j X_i^2 = Y_i^{2} Y_j -(q+q^{-1})Y_iY_jY_i + Y_j Y_i^2=0, 
\qquad
\hbox{if $|i-j|=1$.}
$$
The algebra $U_q(\fgl_N)$ is a Hopf algebra with coproduct and antipode given by
\begin{equation} \label{coproduct-antipode}
\begin{aligned}
& \Delta(L_i) = L_i\otimes L_i, \\
 &\Delta(X_i)= X_i \otimes L_iL_{i+1}^{-1}+ 1 \otimes X_i, \\
& \Delta(Y_i)= Y_i \otimes 1+ L_i^{-1}L_{i+1} \otimes Y_i,
 \end{aligned}
 \qquad\hbox{and}\qquad
\begin{aligned}
& S(L_i) = L_i^{-1}, \\ 
 &S(X_i)=  -X_i L_i^{-1}L_{i+1}, \\
 & S(Y_i)= -L_iL_{i+1}^{-1}Y_i,
 \end{aligned}
\end{equation}
respectively (see  \cite[Section 9.1]{CP}).

As a $\bq(q)$-vector space,  $U_q(\fgl_N)$ has a triangular decomposition 
\begin{equation} \label{rat-triangle}
U_q(\fgl_N) \cong  U_q(\fgl_N)^{<0} \otimes U_q(\fgl_N)^0 \otimes  U_q(\fgl_N)^{>0},
\end{equation}
where the inverse isomorphism is given by multiplication (see \cite[Proposition 9.1.3]{CP}). Here $U_q(\fgl_N)^{<0}$ is the subalgebra generated by the $Y_i$ for $i=1, \ldots, N-1$,  $U_q(\fgl_N)^{>0}$ is the subalgebra generated by the $X_i$ for $i=1, \ldots, N-1,$ and $U_q(\fgl_N)^0$ is the subalgebra generated by the $L_i^{\pm1}$ for $i=1, \ldots, N$.

\subsection{The integral quantum group} \label{res-form}
Let $\cA= \bz[q, q^{-1}]$. For $n,k \in \bz_{>0}$ and $c \in  \bz$, let
\begin{align}
 [n]:= \frac{q^n -q^{-n}}{q-q^{-1}},
 \;
x^{(k)}:=\frac{x^k}{[k] [k-1] \cdots [2][1]},
\; \hbox{and}\;
\left[ 
\begin{array}{c}
x; c \\
k
\end{array}
\right] :=
\prod_{s=1}^k \frac{xq^{c+1-s}-x^{-1}q^{s-1-c}}{q^s-q^{-s}},
\end{align}
in $\bq(q,x)$.
The \emph{restricted integral form} $U_q^{\cA}(\fgl_N)$ is the $\cA$-subalgebra of $U_q(\fgl_N)$ generated by $X_i^{(k)}, Y_i^{(k)}$, $L_i^{\pm 1}$ and $\left[ 
\begin{array}{c}
L_i; c \\
k
\end{array}
\right]$ for $1 \leq i \leq N, c \in \bz, k \in \bz_{>0}$.
As discussed in \cite[Section 6]{Lusztig:1990}, this is an integral form in the sense that 
\begin{equation} \label{int-form}
U_q^\cA(\fgl_N) \otimes_{\cA} \bq(q) = U_q(\fgl_N).
\end{equation} 
As with $U_q(\fgl_N)$, the algebra $U_q^\cA(\fgl_N)$ has a triangular decomposition
\begin{equation}  \label{int-triangle}
U_q^\cA(\fgl_N) \cong U_q^\cA(\fgl_N)^{<0} \otimes U_q^\cA(\fgl_N)^0 \otimes U_q^\cA(\fgl_N)^{>0},
\end{equation}
where the isomorphism is given by multiplication (see \cite[Proposition 9.3.3]{CP}).
In this case, $U_q^\cA(\fgl_N)^{<0}$ is the subalgebra generated by the $Y_i^{(k)}$, $U_q^\cA(\fgl_N)^{>0}$ is the subalgebra generated by the $X_i^{(k)}$, and $U_q^\cA(\fgl_N)^0$ is generated by $L_i^{\pm 1}$ and $\displaystyle \left[ 
\begin{array}{c}
L_i; c \\
k
\end{array}
\right]$
for $1 \leq i \leq N$, $c \in \bz$, and $ k \in \bz_{>0}$.

\subsection{Rational representations} \label{rational-reps} 
The Lie algebra $\fgl_N = M_N(\bc)$ of $N\times N$ matrices has standard basis 
$\{ E_{ij} \ |\ 1\le i,j\le N\},$ where $E_{ij}$ 
is the matrix with $1$ in position $(i,j)$ and $0$ everywhere else. 
Let $\fh= \text{span} \{ E_{11}, E_{22}, \ldots, E_{NN} \}$. Let  $\varepsilon_i \in \fh^*$ be the weight 
of $\fgl_N$ given by $\varepsilon_i(E_{jj})= \delta_{i,j}$. Define
\begin{equation}
\begin{aligned}
&\fh^*_\bz := \{ \lambda=\lambda_1\varepsilon_1+\lambda_2 \varepsilon_2 + \cdots 
+ \lambda_N \varepsilon_N \in \fh^* \ |\ \lambda_1,\ldots, \lambda_N\in \bz\}, \\
&(\fh^*_\bz)^+ 
:= \{ \lambda=\lambda_1\varepsilon_1+\lambda_2 \varepsilon_2 + \cdots + \lambda_N \varepsilon_N
\in \fh_\bz^* \ |\ \lambda_1 \geq \lambda_2 \geq \cdots \geq \lambda_N
\}, \\
&P^+ 
:= \{ \lambda=\lambda_1\varepsilon_1+\lambda_2 \varepsilon_2 + \cdots + \lambda_N \varepsilon_N
\in (\fh_\bz^*)^+ \ |\ \lambda_N\ge 0\}, \\
&R^+ :=\{ \varepsilon_i-\varepsilon_j\ |\ 1 \leq i < j \leq N \}, \\
&Q:= \text{span}_\bz(R^+),  \quad
Q^+:= \text{span}_{\bz_{\ge0}}(R^+), \quad \text{and}\quad Q^-:= \text{span}_{\bz_{\le0}}(R^+).
\end{aligned}
\end{equation}
to be the set of \emph{integral weights}, the set of \emph{dominant integral weights}, the set of \emph{dominant polynomial weights}, the set of \emph{positive roots}, the \emph{root lattice}, the \emph{positive part of the root lattice}, and the the \emph{negative  part of the root lattice}, respectively. 

For an integral weight $\lambda= \lambda_1 \varepsilon_1 + \cdots+\lambda_N\varepsilon_N$, 
the \emph{Verma module} $M(\lambda)$ for $U_q(\fgl_N)$ of highest weight $\lambda$ is 
\begin{equation}
M(\lambda) := U_q(\fgl_N) \otimes_{U_q(\fgl_N)^{\geq 0}} \bq(q)_\lambda,
\end{equation}
where
$\bq(q)_\lambda= \text{span}_{\bq(q)} \{ v_{\lambda} \}$ is the one dimensional vector space over 
$\bq(q)$ with $U_q(\fgl_N)^{\geq 0}$ action given by
\begin{equation}
X_i \cdot v_\lambda =0 \quad\hbox{and}\quad
L_j \cdot v_\lambda =q^{\lambda_j}v_\lambda,
\qquad\hbox{for $1 \leq i \leq N-1$, $1 \leq j \leq N$.}
\end{equation}

\begin{Theorem} \label{Deltaconst} \emph{(see \cite[Chapter 10.1]{CP})}
If $\lambda\in (\fh_\ZZ^*)^+$ then $M(\lambda)$ has a unique finite dimensional quotient 
$\Delta(\lambda)$ and the map $\lambda\mapsto \Delta(\lambda)$ is a bijection between
$(\fh_\ZZ^*)^+$ and the set of isomorphism classes of irreducible finite dimensional $U_q(\fgl_N)$-modules. \end{Theorem}
 
A \emph{singular vector} in a representation of $U_q(\fgl_N)$ is a vector $v$ such that $X_i \cdot v=0$ for all $i$.

 \subsection{Integral representations}
The \emph{integral Verma module} $M^\cA(\lambda)$ is the $U_q^{\cA}(\fgl_N)$-submodule
of $M(\lambda)$ generated by $v_\lambda$. The \emph{integral Weyl module} $\Delta^\cA(\lambda)$ 
is the $U_q^{\cA}(\fgl_N)$-submodule of $\Delta(\lambda)$ generated by $v_\lambda$. 
Using \eqref{int-form} and \eqref{rat-triangle}, 
\begin{equation}
M^{\cA} (\lambda) \otimes_{\cA} \bq(q)= M(\lambda),
\quad \text{and} \quad
\Delta^{\cA}(\lambda) \otimes_{\cA} \bq(q)=\Delta(\lambda).
\end{equation}
In general, $\Delta^{\cA}(\lambda)$ is not irreducible as a $U^\cA_q(\fgl_N)$ module.

\section{Partitions and Fock space} \label{pFock}

We now describe the $v$-deformed Fock space representation of $U_v(\asl_\ell)$ constructed by  Misra and Miwa \cite{MM:1990} following work of Hayashi \cite{Hayashi:1990}. Our presentation largely follows  \cite[Chapter 10]{ariki:2000}. 

\subsection{Partitions}

A partition $\lambda$ is a finite length non-increasing sequence of positive integers. Associated to a partition is its Ferrers diagram. We draw these diagrams as in Figure \ref{partition_bij} so that,
if $\lambda = (\lambda_1,\ldots, \lambda_N)$, then $\lambda_i$ is the number of boxes in row $i$
(rows run southeast to northwest 
$\nwarrow$ 
).  Say that $\lambda$ is contained in $\mu$ if the diagram 
for $\lambda$ fits inside the diagram for $\mu$ and let $\mu/\lambda$ be the collection of boxes 
of $\mu$ that are not in $\lambda$. 
For each box $b \in \lambda$, the \emph{content} $c(b)$ is the horizontal position of $b$
and the \emph{color} $\overline c(b)$ is the residue of $c(b)$ modulo $\ell$.
In Figure \ref{partition_bij}, the numbers $c(b)$ are listed below the diagram.  
The \emph{size} $|\lambda|$ of a partition $\lambda$ is the total number of boxes in its Ferrers diagram.

The set $P^+$ of dominant polynomial weights from Section \ref{rational-reps} is naturally identified with partitions with at most $N$ parts.  
If $\lambda \in P^+$ then
\begin{equation} \label{plusboxes}
\Delta(\lambda) \otimes \Delta(\varepsilon_1) \cong 
\bigoplus_{\tiny \begin{array}{l} 1 \leq k \leq N  \\ \lambda+\varepsilon_k \in P^+
\end{array} } \Delta(\lambda+\varepsilon_k)
\end{equation}
as $U_q(\fgl_N)$-modules. The diagram of $\lambda+\varepsilon_k$ is obtained from the diagram of $\lambda$ by adding a box on row $k$, and $\Delta(\lambda+\varepsilon_k)$ appears in the sum on the right side of \eqref{plusboxes} if and only if $\lambda+\varepsilon_k$ is a partition.
See, for example, \cite[Section 6.1, Formula 6.8]{FH:1991} for the classical statement, and \cite[Proposition 10.1.16]{CP} for the quantum case.

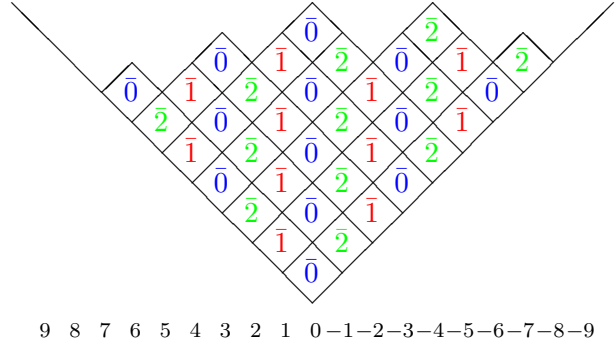
\begin{figure}

\setlength{\unitlength}{0.4cm}
\begin{center}
\begin{picture}(20,12)
\put(0,1){
\begin{picture}(20,11)

\put(10,0){\line(1,1){10}}
\put(9,1){\line(1,1){8}}
\put(8,2){\line(1,1){7}}
\put(7,3){\line(1,1){7}}
\put(6,4){\line(1,1){5}}
\put(5,5){\line(1,1){5}}
\put(4,6){\line(1,1){3}}
\put(3,7){\line(1,1){1}}

\put(10,0){\line(-1,1){10}}
\put(11,1){\line(-1,1){7}}
\put(12,2){\line(-1,1){6}}
\put(13,3){\line(-1,1){6}}
\put(14,4){\line(-1,1){5}}
\put(15,5){\line(-1,1){5}}
\put(16,6){\line(-1,1){3}}
\put(17,7){\line(-1,1){3}}
\put(18,8){\line(-1,1){1}}

\put(3.4,6.7){{ \color{blue} ${\bar 0}$}}

\put(4.4,5.7){{ \color{green} ${\bar 2}$}}
\put(5.4,6.7){{ \color{red} ${\bar 1}$}}
\put(5.4,4.7){{ \color{red} ${\bar 1}$}}
\put(6.4,7.7){{ \color{blue} ${\bar 0}$}}
\put(6.4,5.7){{ \color{blue} ${\bar 0}$}}
\put(6.4,3.7){{ \color{blue} ${\bar 0}$}}
\put(7.4,6.7){{ \color{green} ${\bar 2}$}}
\put(7.4,4.7){{ \color{green} ${\bar 2}$}}
\put(7.4,2.7){{ \color{green} ${\bar 2}$}}
\put(8.4,7.7){{ \color{red} ${\bar 1}$}}
\put(8.4,5.7){{ \color{red} ${\bar 1}$}}
\put(8.4,3.7){{ \color{red} ${\bar 1}$}}
\put(8.4,1.7){{ \color{red} ${\bar 1}$}}
\put(9.4,8.7){{ \color{blue} ${\bar 0}$}}
\put(9.4,6.7){{ \color{blue} ${\bar 0}$}}
\put(9.4,4.7){{ \color{blue} ${\bar 0}$}}
\put(9.4,2.7){{ \color{blue} ${\bar 0}$}}
\put(9.4,0.7){{ \color{blue} ${\bar 0}$}}
\put(10.4,7.7){{ \color{green} ${\bar 2}$}}
\put(10.4,5.7){{ \color{green} ${\bar 2}$}}
\put(10.4,3.7){{ \color{green} ${\bar 2}$}}
\put(10.4,1.7){{ \color{green} ${\bar 2}$}}
\put(11.4,6.7){{ \color{red} ${\bar 1}$}}
\put(11.4,4.7){{ \color{red} ${\bar 1}$}}
\put(11.4,2.7){{ \color{red} ${\bar 1}$}}
\put(12.4,7.7){{ \color{blue} ${\bar 0}$}}
\put(12.4,5.7){{ \color{blue} ${\bar 0}$}}
\put(12.4,3.7){{ \color{blue} ${\bar 0}$}}
\put(13.4,8.7){{ \color{green} ${\bar 2}$}}
\put(13.4,6.7){{ \color{green} ${\bar 2}$}}
\put(13.4,4.7){{ \color{green} ${\bar 2}$}}
\put(14.4,7.7){{ \color{red} ${\bar 1}$}}
\put(14.4,5.7){{ \color{red} ${\bar 1}$}}
\put(15.4,6.7){{ \color{blue} ${\bar 0}$}}
\put(16.4,7.7){{ \color{green} ${\bar 2}$}}

\put(9.6,-1){ $\tiny{{}_{0}}$}
\put(10.1,-1){ $\tiny{{}_{-1}}$}
\put(11.1,-1){ $\tiny{{}_{-2}}$}
\put(12.1,-1){ $\tiny{{}_{-3}}$}
\put(13.1,-1){ $\tiny{{}_{-4}}$}
\put(14.1,-1){ $\tiny{{}_{-5}}$}
\put(15.1,-1){ $\tiny{{}_{-6}}$}
\put(16.1,-1){ $\tiny{{}_{-7}}$}
\put(17.1,-1){ $\tiny{{}_{-8}}$}
\put(18.1,-1){ $\tiny{{}_{-9}}$}

\put(8.6,-1){ $\tiny{{}_{1}}$}
\put(7.6,-1){ $\tiny{{}_{2}}$}
\put(6.6,-1){ $\tiny{{}_{3}}$}
\put(5.6,-1){ $\tiny{{}_{4}}$}
\put(4.6,-1){ $\tiny{{}_{5}}$}
\put(3.6,-1){ $\tiny{{}_{6}}$}
\put(2.6,-1){ $\tiny{{}_{7}}$}
\put(1.6,-1){ $\tiny{{}_{8}}$}
\put(0.6,-1){ $\tiny{{}_{9}}$}

\thicklines

\end{picture}}
\end{picture}
\end{center}
\vspace{0.4cm}

\caption{The partition $(7,6,6,5,5,3,3,1)$ 
with each box containing its color for $\ell =3$. 
The content $c(b)$ of a box $b$ is the horizontal position of $b$ reading right to left. The contents
of boxes are listed beneath the diagram so that $c(b)$ is aligned with all boxes $b$ of that content.
\label{partition_bij}}
\end{figure}

\subsection{The quantum affine algebra}
Let $U'_v(\asl_\ell)$ be the quantized universal enveloping algebra corresponding to the $\ell$-node Dynkin diagram 
\vspace{0.2cm}
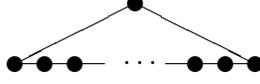
\begin{center}
\setlength{\unitlength}{0.4cm}

\begin{picture}(6,2)

\put(-1,0){\circle*{0.5}}
\put(0,0){\circle*{0.5}}
\put(1,0){\circle*{0.5}}
\put(5,0){\circle*{0.5}}
\put(6,0){\circle*{0.5}}
\put(7,0){\circle*{0.5}}

\put(3,2){\circle*{0.5}}

\put(3,2){\line(2,-1){4}}
\put(3,2){\line(-2,-1){4}}

\put(-1,0){\line(1,0){1}}
\put(0,0){\line(1,0){1}}
\put(1,0){\line(1,0){1}}
\put(4,0){\line(1,0){1}}
\put(5,0){\line(1,0){1}}
\put(6,0){\line(1,0){1}}

\put(2.6,0){\ldots}

\end{picture}
\end{center}
\vspace{0.1cm}
More precisely,  $U'_v(\asl_\ell)$ is the algebra generated by $ E_{\bar i}, F_{\bar i}, 
K_{\bar i}^{\pm1}$, for $\bar i \in \bz/\ell \bz$, with relations
$$K_{\bar i}K_{\bar j}=K_{\bar j}K_{\bar i}, \quad 
K_{\bar i}K_{\bar i}^{-1}= K_{\bar i}^{-1}K_{\bar i} =1,
\qquad
E_{\bar i} F_{\bar j} - F_{\bar j} E_{\bar i} 
= \delta_{\bar i,\bar j} \frac{K_{\bar i} - K_{\bar i}^{-1}}{v-v^{-1}},
$$ 
\begin{equation}
K_{\bar i} E_{\bar j} K_{\bar i}^{-1}=
\begin{cases} v^2 E_{\bar j}, \quad \text{ if } \bar i=\bar j, \\
v^{-1}E_{\bar j}, \quad \text{ if } \bar i=\bar j\pm1, \\
E_{\bar j} \quad \text{ otherwise};
\end{cases}
\qquad
K_{\bar i} F_{\bar j} K_{\bar i}^{-1}=
\begin{cases} v^{-2}F_{\bar j}, \quad \text{ if } \bar i=\bar j, \\
v F_{\bar j}, \quad \text{ if } \bar i=\bar j\pm1, \\
F_{\bar j}, \quad \text{ otherwise};
\end{cases}
\end{equation}
$$E_{\bar i}E_{\bar j} = E_{\bar j} E_{\bar i}
\quad\hbox{and}\quad 
F_{\bar i} F_{\bar j}=F_{\bar j}F_{\bar i},
\qquad\hbox{if $|\bar i-\bar j| \geq 2$,}$$
$$
E_{\bar i}^{2} E_{\bar j} -(v+v^{-1}) E_{\bar i}E_{\bar j}E_{\bar i} + E_{\bar j} E_{\bar i}^2 
= F_{\bar i}^{2} F_{\bar j} -(v+v^{-1}) F_{\bar i}F_{\bar j}F_{\bar i} + F_{\bar j} F_{\bar i}^2=0, 
\qquad
\hbox{if $|\bar i-\bar j|=1$.}
$$
See \cite[Definition Proposition 9.1.1]{CP}. 
The algebra $U_v'(\asl_\ell)$ is the quantum group corresponding
to the non-trivial central extension 
$\asl_\ell' = \fsl_\ell[t,t^{-1}] \oplus \bc c$ of the algebra of polynomial loops in $\mathfrak{sl}_\ell$.

\subsection{Fock space} \label{MM_section}
Define \emph{$v$-deformed Fock space} to be the $\bq(v)$ vector space
$\Fock_v$ with basis 
$\{ |\lambda \rangle \ |\  \lambda \text{ is a partition} \}$.
Our $\Fock_v$ is only the charge $0$ part of Fock space described in \cite{KMS:1995}.
Fix $\bar i \in \bz/\ell \bz$ and partitions $\lambda \subseteq \mu$ such that $\mu / \lambda$ is a single box. Define
\begin{equation}
\begin{array}{l}
\text{$A_{\bar i}(\lambda) \hspace{-0.1cm} := \hspace{-0.1cm} \{\text{boxes } b  :  b \notin \lambda, b \text{ has color } 
\bar i \text{ and } \lambda \cup b \text{ is a partition} \},$} \\

 \text{$R_{\bar i}(\lambda) \hspace{-0.1cm} := \hspace{-0.1cm} \{\text{boxes } b: b \in \lambda,  b \text{ has color } 
\bar i \text{ and } \lambda \backslash b \text{ is a partition} \},$} \\

 \text{$N_{\bar i}^l(\mu / \lambda) \hspace{-0.1cm} :=  \hspace{-0.1cm} | \{ b \in R_{\bar i}(\lambda) :   \hspace{-0.05cm}  b \text{ to the left of } \mu / \lambda \}|
\hspace{-0.1cm} - \hspace{-0.1cm}
| \{ b \in A_{\bar i}(\lambda) :   b \text{ to the left of } \mu / \lambda \}|
,$} \\

\text{$N_{\bar i}^r(\mu / \lambda) \hspace{-0.1cm} := \hspace{-0.1cm} | \{ b \in R_{\bar i}(\lambda) : b \text{ to the right of } \mu / \lambda \}|
\hspace{-0.1cm} - \hspace{-0.1cm}
| \{ b \in A_{\bar i}(\lambda) : b \text{ to the right of  } \mu / \lambda \}|$}
\end{array}
\end{equation}
to be  the set of \emph{addable boxes of color $\bar i$}, the set of \emph{removable boxes of color $\bar i$}, the \emph{left removable-addable difference}, and the \emph{right removable-addable difference}, respectively.

\begin{Theorem}  \emph{(see \cite[Theorem 10.6]{ariki:2000})}
\label{MM_Fock_th}
There is an action of $U'_v(\asl_\ell)$ on $\Fock_v$ determined by
\begin{align} \label{MM-a}
E_{\bar i}|\lambda\rangle &:= 
\sum_{ \overline{c}(\lambda / \mu)={\bar i}} v^{-N_{\bar i}^r(\lambda / \mu )} |\mu\rangle
\qquad \quad  \hbox{and}
& F_{\bar i}|\lambda\rangle &:= \sum_{ \overline{c}(\mu / \lambda)={\bar i}} v^{N_{\bar i}^l(\mu / \lambda )} |\mu\rangle,
\end{align}
where $\overline c(\lambda / \mu)$ denotes the color of $\lambda / \mu$ and the sum is over 
partitions $\mu$ which differ from $\lambda$ by removing (respectively adding) a 
single $\bar i$-colored box.
\end{Theorem}

As a $U'_v(\asl_\ell)$-module, $\Fock_v$ is isomorphic to an infinite direct sum of copies of the basic representation $V(\Lambda_0)$. Using the grading of $\Fock_v$ where $|\lambda \rangle$ has degree $|\lambda|$, the highest weight vectors in $\Fock_v$ occur in degrees divisible by $\ell$, and the number of highest weight vectors in degree $\ell k$ is the number of partitions of $k$. Then $\Fock_v\cong V(\Lambda_0) \otimes \bc[x_1, x_2, \ldots]$, 
where $x_k$ has degree $\ell k$, and $U'_v(\asl_\ell)$ acts trivially on the second factor (see \cite[Prop.\ 2.3]{KMS:1995}). Note that we are working with the `derived' quantum group $U_v'(\asl_\ell)$, not the `full' quantum group $U_v(\asl_\ell)$, which is why there are no $\delta$-shifts in the summands of $\Fock_v$.

\begin{Comment} Comparing with \cite[Chapter 10]{ariki:2000}, 
our $N_{\bar i}^l(\mu/\lambda)$ is equal to Ariki's $-N_{\bar i}^a(\mu /\lambda)$ and 
our $N_{\bar i}^r(\mu/\lambda)$ is equal to Ariki's $-N_{\bar i}^b(\mu /\lambda)$. However, these numbers play a slightly different role in Ariki's work, which is explained by a different choice of conventions. 
\end{Comment}

\section{Universal Verma modules} \label{UV-section}

The purpose of this section is to construct a family of representations which are 
universal Verma modules in the sense that each can be ``evaluated" to obtain any given Verma module. This notion was defined by Kashiwara \cite{Kashiwara:1985} in the classical case, and was studied in the quantum case by Kamita \cite{Kamita}. 

\subsection{Rational universal Verma modules} \label{rat-UV}  \label{twisted-UV}

Let $\KK:= \bq(q,z_1,z_2, \ldots, z_N)$.  This field is isomorphic to the field of fractions of $U_q(\fgl_N)^0$ via the map 
\begin{equation}\label{psidefn}
\psi: U_q(\fgl_N)^0 \rightarrow \KK
\qquad\hbox{defined by}\qquad \psi(L_i^{\pm 1}) = z_i^{\pm 1}.
\end{equation}
For each $\mu \in \mathfrak{h}_\bz^*$, define a $\bq(q)$-linear 
automorphism $\sigma_\mu\colon \KK \to \KK$ by
\begin{equation}\label{sigmadefn}
\begin{aligned}
\sigma_\mu(z_i):= q^{(\mu, \varepsilon_i)} z_i,
\qquad\hbox{for $1\le i\le N$,}
\end{aligned}
\end{equation}
where $(\cdot, \cdot)$ is the inner product on $\mathfrak{h}_\bz^*$ defined by $(\varepsilon_i, \varepsilon_j)= \delta_{i,j}$.
Let $ \KK_\mu = \text{span}_\KK \{ v_{\mu+} \}$ be the one dimensional vector space over $\KK$ with
basis vector $v_\mu^+$ and 
$U_q(\fgl_N)^{\geq 0}$ action given by 
\begin{equation}
X_i \cdot v_{\mu +}=0,\quad\hbox{for $1 \leq i \leq N-1$, \quad and}
\qquad
a \cdot v_{\mu+}=\sigma_\mu(\psi (a)) v_{\mu+},\quad\hbox{for $a \in U_q(\fgl_N)^{0}$.}
\end{equation}

The {\it $\mu$-shifted rational universal Verma module} ${}^\mu \widetilde M$ is the $U_q(\fgl_N)$-module 
\begin{equation} \label{ratUV-def}
{}^\mu \widetilde M:= U_q(\fgl_N) \otimes_{U_q(\fgl_N)^{\geq 0}} \KK_\mu. 
\end{equation}
The universal Verma module ${}^\mu \widetilde M$ is actually a module over $U_q(\fgl_N) \otimes_{U_q(\fgl_N)^0} \widetilde U_q(\fgl_N)^0 $, where $ \widetilde U_q(\fgl_N)^0$ is the field of fractions of $U_q(\fgl_N)^0$. However, if we identify $ \widetilde U_q(\fgl_N)^0$ with 
$\KK$ using the map $\psi$, the action of $\widetilde U_q(\fgl_N)^0 $ on ${}^\mu \widetilde M$ 
is not by multiplication, but rather is twisted by the automorphism $\sigma_\mu$. It is to keep track of 
the difference between the action of $ U_q(\fgl_N)^0$ and multiplication that we use different 
notation for the generators of $\KK$ and $U_q(\fgl_N)^0$ (that is, $z_i$ versus $L_i$).

\subsection{Integral universal Verma modules} \label{int-UV}

The field $\KK$ contains an $\cA$-subalgebra 
\begin{equation}
\cR \quad\hbox{generated by}\qquad
z_i^{\pm 1}\quad\hbox{and}\quad
\left[ 
\begin{array}{c}
z_i; c \\
k
\end{array}
\right]
\qquad(1 \leq  i \leq N, c \in \bz, k \in \bz_{>0}),
\end{equation}
which is isomorphic to $U_q^\cA(\fgl_N)^0$ via the restriction of the map $\psi$ in
\eqref{psidefn}.
The {\it integral universal Verma module} ${}^\mu \widetilde M^\cR$ is the 
$U_q^\cA(\fgl_N)$-submodule 
of ${}^\mu \widetilde M$ generated by $v_{\mu+}$. 
By restricting \eqref{ratUV-def}, 
\begin{equation}
{}^\mu \widetilde M^\cR= U_q^\cA(\fgl_N) \otimes_{U_q^\cA(\fgl_N)^{\geq 0}} \cR_\mu,
\end{equation}
where $\cR_\mu$ is the $\cR$-submodule of $ \KK_\mu$ spanned by $v_{\mu+}$. In particular, ${}^\mu \widetilde M^\cR$ is a free $\cR$-module. 

\subsection{Evaluation} \label{eval}
Let $\se_\lambda^\cR: \cR \rightarrow \cA$ be the map defined by 
\begin{equation} \label{ARE}
\se^\cR_\lambda(z_i) =q^{(\lambda, \varepsilon_i)}  \qquad\hbox{and}\qquad
\se^\cR_\lambda
\left[ 
\begin{array}{c}
z_i; c \\
n
\end{array}
\right]  =
\left[ 
\begin{array}{c}
q^{(\lambda, \varepsilon_i)}; c \\
n
\end{array}
\right],
\end{equation}
where $(\cdot, \cdot)$ is the inner product on $\mathfrak{h}^*$ defined by $(\varepsilon_i, \varepsilon_j)=\delta_{i,j}$. 

There is a surjective $U_q^\cA(\fgl_N)$-module homomorphism
``evaluation at $\lambda$''
\begin{equation}
\ev_\lambda: {}^\mu \widetilde M^\cR \rightarrow M^\cA(\mu+\lambda)
\quad\hbox{defined by}\quad
\ev_\lambda(a \cdot v_{\mu+}):= a \cdot v_{\mu+\lambda}, \quad \text{for all $a \in U^\cA_q(\fgl_N)$.}
\end{equation} 

For fixed $\lambda$, the maps  $\ev_\lambda^\cR$ and $\ev_\lambda$ extend to a map from the subspace of $\KK$ and ${}^\mu \widetilde M ={}^\mu \widetilde M^\cR \otimes_\cR \KK$  respectively where no denominators evaluate to 0. Where it is clear we denote both these extended maps by $\ev_\lambda$.
 
 \begin{Example} Computing the action of $L_i$ on $v_{\mu+}$ and $v_{\mu+\lambda}$, 
\begin{equation}
L_i \cdot  v_{\mu+}= q^{(\mu, \varepsilon_i)} z_i v_{\mu+},
\qquad\hbox{and}\qquad
\begin{array}{rl}
L_i \cdot  v_{\mu+\lambda}
&= \ev_\lambda (q^{(\mu, \varepsilon_i)} z_i ) v_{\mu+\lambda} \\
&= q^{(\mu, \varepsilon_i)}q^{(\lambda, \varepsilon_i)} v_{\mu+\lambda} = q^{(\mu+\lambda, \varepsilon_i)} v_{\mu+\lambda}.
\end{array}
\end{equation}
\end{Example}

\subsection{Weight decompositions} \label{dec}

Let $\widetilde V$ be a $U_q(\fgl_N) \otimes_\cA \cR$-module. For each $\nu \in \mathfrak{h}_\bz^*$, we define the \emph{$\nu$-weight space} of $\widetilde{V}$ to be
\begin{equation}
\widetilde{V}_\nu :=  \{ v \in  \widetilde V : L_i \cdot v = q^{(\nu, \varepsilon_i)} z_i v \}.
\end{equation}
The universal Verma module ${}^\mu \widetilde M^\cR$ is a  $U_q(\fgl_N) \otimes_\cA \cR$-module, where the second factor acts as multiplication.  The weight space ${}^\mu \widetilde M_\eta \neq 0$ if and only if $\eta=\mu-\nu$ with $\nu$ in the positive part $Q^+$ of the root lattice. These non-zero weight spaces and the weight decomposition of ${}^\mu \widetilde M$ can be described explicitly by
\begin{equation}
{}^\mu \widetilde M_{\mu-\nu}^\cR = U^\cA_q(\fgl_N)^{< 0}_{-\nu} \cdot \cR_\mu
\qquad\hbox{and}\qquad
{}^\mu \widetilde M^\cR = \bigoplus_{\nu\in Q^+} {}^{\mu} \widetilde M^\cR_{\mu-\nu}.
\end{equation}
Here $U_q^\cA(\fgl_N)^{< 0}_{-\nu}$ is defined using the grading of $U_q(\fgl_N)^{<0}$ with
$F_i\in U_q(\fgl_N)^{<0}_{-(\varepsilon_i-\varepsilon_{i+1})}$. 

\subsection{Tensor products} \label{universal-tensor} Let $\widetilde V$ be a $U_q^\cA(\fgl_N) \otimes_\cA \cR$-module and $W$ a $U_q^\cA(\fgl_N)$-module. The tensor product $\widetilde V \otimes_\cA W$ is a $U_q^\cA(\fgl_N) \otimes_\cA \cR$-module, where the first factor acts via the usual coproduct and the second factor  acts by multiplication on $\widetilde V$. 
In the case when $\widetilde V$ and $W$ both have weight space decompositions, the weight spaces of $\widetilde V \otimes_\cA W$ are
\begin{equation}
(\widetilde V \otimes_\cA W)_\nu = \bigoplus_{\gamma+\eta=\nu} \widetilde V_\gamma \otimes_\cA W_\eta.
\end{equation}

We also need the following:


\begin{Proposition} \label{rat_dec}
The tensor product of a universal Verma module with a Weyl module satisfies
 \begin{equation} 
\left({}^\mu \widetilde M^\cR \otimes_\cA \Delta^\cA(\nu) \right)\otimes_\cR \KK
\cong \left( \bigoplus_{\gamma} ({}^{\mu + \gamma} \widetilde M^\cR)^{\oplus \dim \Delta^\cA(\nu)_\gamma} \right) \otimes_\cR \KK.
\end{equation}
\end{Proposition}

\begin{proof}
Fix $\nu \in P^+$.  In general, $M(\lambda+\mu) \otimes \Delta(\nu)$ has a Verma filtration
(see, for example, \cite[Theorem 2.2]{Jantzen:1979}) and
if $\lambda+\mu+\gamma$ is dominant for all $\gamma$ such that 
$\Delta(\nu)_\gamma\ne 0$ then
\begin{equation} 
M(\lambda+\mu) \otimes \Delta(\nu) 
\cong \bigoplus_{\gamma} M(\lambda+\mu+\gamma)^{\oplus \dim \Delta(\nu)_\gamma},
\end{equation}
which can be seen by, for instance, taking central characters.
The proposition follows since this is true for a Zariski dense set of weights $\lambda$.
\end{proof}

\section{The Shapovalov form and the Shapovalov determinant} \label{Sform}

\subsection{The Shapovalov form} \label{S_over_A}

The \emph{Cartan involution} $\omega: U_q(\fgl_N) \rightarrow U_q(\fgl_N)$ is the $\bq(q)$-algebra anti-involution of $U_q(\fgl_N)$ defined by 
\begin{equation}
\omega(L_i^{\pm 1})=L_i^{\pm 1}, 
\qquad
\omega(X_i)= Y_iL_iL_{i+1}^{-1}, 
\qquad
\omega(Y_i) = L_i^{-1}L_{i+1} X_i. 
\end{equation}
The map $\omega$ is also a co-algebra involution. 
An \emph{$\omega$-contravariant} form on a $U_q(\fgl_N)$-module $V$ 
is a symmetric bilinear form $(\cdot, \cdot)$ such that
\begin{equation} \label{contrav}
(u,a \cdot v)=(\omega(a) \cdot u, v),\qquad
\hbox{for $u, v \in V$ and $a \in U_q(\fgl_N)$.}
\end{equation}

It follows by the same argument used in the classical case \cite{Shapovalov:1972} that there is an 
$\omega$-contravariant form (the Shapovalov form) on each Verma module $M(\lambda)$ 
and this is unique up to rescaling. The radical of $(\cdot, \cdot)$ is the maximal proper submodule of $M(\lambda)$, so $\Delta(\lambda)=M(\lambda)/\mathrm{Rad}( \cdot, \cdot)$ for all $\lambda \in P^+$. In particular, $(\cdot, \cdot)$ descends to an $\omega$-contravariant form on $\Delta(\lambda)$. 

Since $\omega$ fixes $U_q^\cA(\fgl_N) \subseteq U_q(\fgl_N)$, there is a well defined 
notion of an $\omega$-contravariant form on a $U_q^\cA(\fgl_N)$ module. In particular, 
the restriction of the Shapovalov form on $\Delta(\lambda)$ to $\Delta^\cA(\lambda)$ 
is $\omega$-contravariant.

\subsection{Universal Shapovalov forms}
There are surjective maps of $\cA$-algebras $p_-:  U_q^\cA(\fgl_N)^{<0} \rightarrow \bq(q)$  and  $p_+:  U_q^\cA(\fgl_N)^{>0} \rightarrow \bq(q)$  defined by $p_-(F_i)=0$ and  $p_+(E_i)=0$, for $1 \leq i \leq N$. Using the triangular decomposition \eqref{int-triangle}, there is an $\cA$-linear surjection 
\begin{equation}
\pi_0 := p_- \otimes \Id \otimes p_+: U_q^\cA(\fgl_N) \cong U_q^\cA(\fgl_N)^{<0} \otimes_\cA U^\cA_q(\fgl_N)^0 \otimes_\cA U^\cA_q(\fgl_N)^{>0} \rightarrow U_q^\cA(\fgl_N)^0.
\end{equation}
The {\it standard universal Shapovalov form} is the $\cR$-bilinear form $(\cdot, \cdot)_{{}^\mu \widetilde M^\cR}: {}^\mu \widetilde M^\cR \otimes {}^\mu \widetilde M^\cR \rightarrow \cR$ defined by 
\begin{equation}
(a_1 \cdot  v_{\mu+},a_2 \cdot v_{\mu+})_{{}^\mu \widetilde M^\cR}
=  \big(\sigma_\mu \circ \psi \circ \pi_0\big) (\omega(a_2) a_1)
\end{equation}
for all  $a_1,a_2 \in  U_q^\cR(\fgl_N)^{<0}$. 
Here $\psi$ and $\sigma_\mu$ are as in \eqref{psidefn} and \eqref{sigmadefn}.
Since 
\begin{equation}
(a_1a_2 \cdot  v_{\mu+},a_3 \cdot v_{\mu+})_{{}^\mu \widetilde M^\cR}
=  \big(\sigma_\mu \circ \psi \circ \pi_0\big) (\omega(a_2)\omega(a_1)a_3)
=(a_2 \cdot  v_{\mu+},\omega(a_1)a_3 \cdot v_{\mu+})_{{}^\mu \widetilde M^\cR}
\end{equation}
for $a_1,a_2,a_3 \in U_q(\fgl_N),$
the form 
$(\cdot, \cdot)_{{}^\mu \widetilde M^\cR}$ is $\omega$-contravariant. As with the usual Shapovalov form, distinct weight spaces are orthogonal, where weight spaces are defined as in Section \ref{dec}.

Evaluation
at $\lambda$ gives an $\cA$-valued $\omega$-contravariant form 
$(\cdot, \cdot)_{M^\cA(\mu+\lambda)}$ on $M^\cA(\mu+\lambda)$ by
 \begin{equation} \label{eve}
(\ev_\lambda(u_1), \ev_\lambda(u_2))_{M^\cA(\mu+\lambda)} 
= \ev_\lambda \left((u_1,u_2)_{{}^\mu \widetilde M^\cR}\right),
\qquad\hbox{for $u_1,u_2 \in {}^\mu \widetilde M^\cR$.}
\end{equation}

The form $(\cdot, \cdot)_{{}^\mu \widetilde M^\cR}$ can be extended by linearity to an $\omega$-contravariant form $(\cdot, \cdot)_{{}^\mu \widetilde M}$ on ${}^\mu \widetilde M$. 

\subsection{The Shapovalov determinant} \label{SD}
Let $\widetilde V$ be a $(U_q^\cA(\fgl_N) \otimes_\cA \cR)$-module with a chosen 
$\omega$-contravariant form.
Let $B_\eta$ be an $\cR$ basis for the $\eta$-weight space $\widetilde V_\eta$ of $\widetilde V$. 
Let $\det \widetilde V_{B_\eta}$ be the determinant of the form 
evaluated on the basis $B_\eta$. Changing the basis $B_\eta$ changes the determinant by a unit
in $\cR$ and we sometimes write $\det\widetilde V_\eta$ to mean the determinant calculated on an unspecified basis ($\det\widetilde V_\eta$ which is only defined up to multiplication by unit in $\cR$). 
The {\it Shapovalov determinant}  is
\begin{equation}
\det \widetilde M^\cR_\eta := \det((b_i,b_j)_{\widetilde M^\cR})_{b_i,b_j\in B_\eta}.
\end{equation}

Define the {\it partition function} $p\colon \fh^*\to \ZZ_{\ge 0}$ by 
\begin{equation}
p(\gamma):= \dim M(0)_\gamma.
\end{equation}
Then
$p(\gamma)= \dim M(\lambda)_{\gamma+\lambda}$
for any $\lambda$, and $\eta \not\in Q^-$ implies that $p(\eta)=0$ and $\det \widetilde M^\cR_\eta=1$.

\begin{Theorem} \label{fulldet} \emph{(see \cite[Proposition 1.9A]{KD:1991}, \cite[Theorem 3.4]{KL:1997}, \cite{Shapovalov:1972})}
For any weight $\eta$, 
\begin{equation} 
\det \widetilde M^\cR_\eta 
= c_\eta \prod_{\tiny \begin{array}{c} 1 \leq i < j \leq N \\ m >0 \end{array}} 
{\Big( }z_iz_j^{-1} - q^{2m+2i-2j} z_i^{-1}z_j {\Big)}^{p(\eta+m\varepsilon_i-m\varepsilon_j)},
\end{equation}
where $c_\eta$ is a unit in $\cR \otimes_\cA \bq(q) = \bq(q)[ z_1^{\pm1}, \ldots, z_N^{\pm 1}]$.
\end{Theorem}

\begin{Proposition} \label{mu-s}
Fix $\mu, \eta \in \fh_\ZZ^*$ with $
\eta-\mu \in Q^-.$ Choose an $\cA$-basis $B_{\eta-\mu}$ for $U_q^\cA(\fgl_N)_{\eta-\mu}.$ Consider the $\cR$-bases 
$\widetilde B_{\eta-\mu} := \{ b \cdot v_+ \ |\ b \in B_{\eta-\mu} \}$ for $\widetilde M^{\cR}_{\eta-\mu}$ 
and 
${}^\mu \widetilde B_\eta : = \{ b \cdot {v_{\mu+}} \ |\  b \in B_{\eta-\mu} \}$ 
for $ {}^\mu \widetilde M^\cR_\eta $. Then 
$\det {}^\mu \widetilde M^\cR_{({}^\mu \widetilde B_\eta)}
= \sigma_\mu \big(\det \widetilde M^\cR_{\widetilde B_{\eta-\mu}}\big)$.
\end{Proposition}

\begin{proof}
For $b,b' \in B_{\eta-\mu}$,
\begin{align}
(b \cdot v_{\mu+},b' \cdot v_{\mu+})_{{}^\mu \widetilde M^\cR} 
= \sigma_\mu \circ \psi  \circ \pi_0 (\omega(b') b) 
= \sigma_\mu \big( (b \cdot v_{0+},b' \cdot  v_{0+})_{\widetilde M^\cR}\big).
\end{align}
The result follows by taking determinants.
\end{proof}

\subsection{Contravariant forms on tensor products} \label{SonT}

If $V$ and $W$ are  $U_q^\cA(\fgl_N)$-modules with $\omega$-contravariant forms 
$(\cdot, \cdot)_V$ and $(\cdot, \cdot)_W$,
define an $\cA$-bilinear form $(\cdot,\cdot)_{W \otimes V}$ by 
 $(w_1\otimes v_1, w_2\otimes v_2)_{W \otimes V}
 =  (w_1,w_2)_W(v_1,v_2)_{V}$.
Similarly, for 
a $U_q^\cA(\fgl_N) \otimes_\cA \cR$ module $\widetilde W$ with $\cR$-bilinear 
$\omega$-contravariant form $(\cdot, \cdot)_{\widetilde W}$, define
 a $\cR$-bilinear form $(\cdot, \cdot)_{\widetilde W \otimes_{\bq(q)} V}$ 
 on $\widetilde W \otimes_{\bq(q)} V$ by
\begin{equation}\label{prodformdefn}
(u_1 \otimes v_1, u_2 \otimes v_2)_{\widetilde W \otimes_{\bq(q)} V} 
= (u_1,u_2)_{\widetilde W} (v_1,v_2)_V.
\end{equation}
Since $\omega$ is a coalgebra involution (i.e., $\Delta(\omega(a))= (\omega \otimes \omega) \Delta(a)$, for $a \in U_q(\fgl_N)$),
the forms $(\cdot,\cdot)_{V \otimes W}$ and 
$(\cdot, \cdot)_{{}^\mu \widetilde M \otimes_{\bq(q)} V}$
are $\omega$-contravariant. 

In the case when $\widetilde W= {}^\mu \widetilde M^\cR$, evaluation of the 
$\omega$-contravariant form $(\cdot, \cdot)_{{}^\mu \widetilde M^\cR \otimes_\cA V}$
at $\lambda$ gives an $\omega$-contravariant form 
$(\cdot, \cdot)_ {M^\cA(\mu+\lambda) \otimes_\cA V}$:
\begin{equation} \label{ev-tensor}
\begin{aligned}
(u_1 \otimes v_1, u_2 \otimes v_2)_ {M^\cA(\mu+\lambda) \otimes_\cA V} &= \ev_\lambda \left((u_1 \otimes v_1, u_2 \otimes v_2)_ {{}^\mu \widetilde M^\cR \otimes_\cA V} \right) \\& = (\ev_\lambda(u_1) \otimes v_1, \ev_\lambda(u_2) \otimes v_2)_{M(\mu+\lambda) \otimes_\cA V},
\end{aligned}
\end{equation}
for $u_1, u_2 \in {}^\mu \widetilde M$ and $v_1, v_2 \in V$. 
As in Section \ref{eval}, this evaluation can be extended to the $\cA$-submodule of the
rational module where no denominators evaluate to zero.

\section{The Misra-Miwa formula for $F_{\bar i}$ from $U_q^\cA(\fgl_N)$ representation theory}  \label{makeops}

Let us prepare the setting for our main result (Theorem 6.1). Fix $\ell \geq 2$ and a partition $\lambda$. Let $N$ a positive integer greater than the 
number of parts of $\lambda$. All calculations below are in terms of representations 
of $U_q^\cA(\fgl_N)$. 

$\bullet$ Let $V=\Delta^\cA(\varepsilon_1)$ be the standard $N$-dimensional module. Since $\Delta^\cA(\lambda) \otimes_\cA \bq(q)= \Delta(\lambda)$, Equation \eqref{plusboxes} implies
\begin{equation} \label{tens}
 \left( \Delta^{\cA}(\lambda) \otimes_\cA V \right) \otimes_{\cA} \bq(q)  \simeq \bigoplus \Delta^{\cA}(\lambda+\varepsilon_{k_j})  \otimes_{\cA} \bq(q),
\end{equation}
where the sum is over those indices $1 = k_1 < k_2 < \cdots < k_{m_\lambda} \leq N$ for which $\lambda+\varepsilon_{k_j}$ is a partition. 
For ease of notation let  $\mu^{(j)}= \lambda+\varepsilon_{k_j}$.

$\bullet$ Fix an $\cA$-basis $\{ v_1, \ldots, v_N \}$ of $V$ where $v_k$ has weight $\varepsilon_k$ and $Y_i(v_k) = \delta_{i,k} v_{k+1}$. Recursively define singular weight vectors $v_{\mu^{(j)}}$ in $\left(\Delta^\cA(\lambda) \otimes V \right) \otimes_\cA \bq(q)$ by:
\begin{enumerate}
\item $v_{\mu^{(1)}}= v_\lambda \otimes v_1.$

\item For each $k$, the submodule $W_k$ of $(\Delta(\lambda) \otimes_\cA V) \otimes_\cA \bq(q)$ generated by $\{ v_\lambda \otimes v_i\ |\ 1 \leq i \leq k \}$ contains all weight vectors of $(\Delta(\lambda) \otimes_\cA V) \otimes_\cA \bq(q)$ of weight greater than or equal to $\lambda+\varepsilon_k$. Thus, using \eqref{tens}, for each $1 \leq j \leq m_\lambda$ there is a 
one-dimensional space of singular vectors of weight $\mu^{(j)}$ in $W_{k_j}$, 
and this is not contained in $W_{k_{j-1}}$ 
(since $k_j> k_{j-1}$). 
This implies that there unique singular vector $v_{\mu^{(j)}}$ of weight $\mu^{(j)}$ in  
\begin{equation} \label{eqs2}
v_\lambda \otimes v_{k_j} + \bigoplus_{1 \leq i <j} U_q(\fgl_N)  v_{\mu^{(i)}} 
\subseteq \left( \Delta^\cA(\lambda)\otimes_\cA V \right) \otimes_\cA \bq(q),
\end{equation}
where we recall that $U_q(\fgl_N)= U_q^\cA(\fgl_N) \otimes_\cA \bq(q)$.
\end{enumerate}

$\bullet$ There is a unique $\omega$-contravariant form on $\Delta^\cA(\lambda)$ normalized so that $(v_\lambda, v_\lambda)=1$ and a unique $\omega$-contravariant form on $V$ normalized so that $(v_1,v_1)=1$. As in section \ref{SonT}, define a $\omega$-contravariant form on $\left( \Delta^{\cA}(\lambda) \otimes_\cA V \right) \otimes_{\cA} \bq(q)$ by $(u_1 \otimes w_1, u_2 \otimes w_2)= (u_1,u_2)(w_1,w_2)$. For each $ 1\leq j \leq m_\lambda$, define an element $r_{j}(\lambda) \in \bq(q)$ by
\begin{equation}
r_{j}(\lambda) := (v_{\mu^{(j)}}, v_{\mu^{(j)}}).
\end{equation}

\begin{Theorem} \label{main} The Misra-Miwa operators $F_{\bar i}$ from Section \ref{MM_section} satisfy
\begin{equation} \label{Fdeftwo}
\displaystyle F_{\bar i} |\lambda\rangle = \sum_{\bar c(b^{(j)})= \bar i}v^{\val_{\phi_{2\ell}} r_{j}(\lambda)} |\mu^{(j)}\rangle,
\end{equation}
where $b^{(j)}$ is the box $\mu^{(j)} / \lambda$, $\bar c (b^{(j)})$ is the color of box $b^{(j)}$ as in Figure \ref{partition_bij}, $\phi_{2\ell}$ is the $2\ell^{th}$ cyclotomic polynomial in $q$ and $\val_{\phi_{2\ell}} r$ is the number of factors of $\phi_{2\ell}$ in the numerator of $r$ minus the number of factors of $\phi_{2\ell}$ in the denominator of $r$.
\end{Theorem}

The proof of Theorem \ref{main} will occupy the rest of this section. We will first prove a similar statement, Proposition \ref{the_vals2}, where the role of the Weyl modules is played by the universal Verma modules from Section \ref{UV-section}. For ease of notation, let $\widetilde M^\cR$ denote the module ${}^0 \widetilde M^\cR$ from section \ref{int-UV}.

\begin{Definition} \label{make-the-high-vectors} \label{mhv}
Recursively define singular weight vectors $v_{\varepsilon_k+} \in \left( \widetilde M^\cR \otimes_\cA V \right) \otimes_\cR \KK$ and elements $s_k \in \KK$ for $1 \leq k \leq N$ by
\begin{enumerate}
\item $v_{\varepsilon_1+}= v_+ \otimes v_1.$

\item Since $\{v_+ \otimes v_j\ |\ 1\le j\le N\}$ generates  $\widetilde M^\cR \otimes_\cA V$ as a $U_q^\cA(\fgl_N)^{\leq 0}$ module, Proposition \ref{rat_dec} implies that, for each $1 \leq k \leq N$, there is a unique singular vector $v_{\varepsilon_{k}\tiny +}$ in  
$\displaystyle v_+ \otimes v_k 
+ \bigoplus_{1 \leq j <k} U_q^\KK(\fgl_N)  v_{\varepsilon_j+} 
\subseteq \left( \widetilde M^\cR \otimes_\cA V \right) \otimes_\cR \KK$, where 
$U_q^\KK(\fgl_N):= U_q(\fgl_N) \otimes_{\bq(q)} \KK$ and the factor of $\KK$ acts by multiplication on $\widetilde M^\cR$. 
\end{enumerate}
Let $s_k = (v_{\varepsilon_k\tiny+}, v_{\varepsilon_k \tiny+})$. 
\end{Definition}

The $s_k$ are quantized versions of the Jantzen numbers 
first calculated in \cite[Section 5]{Jantzen:1974} and quantized in \cite{Wiesner}. It follows immediately from the definition that $s_1=1$.

\begin{Lemma} \label{equal-dets}
For any weight $\eta$, up to multiplication by a power of $q$,
\begin{equation} 
\prod_{1 \leq k \leq N} s_k^{p(\eta-\varepsilon_k)}  =
\prod_{1 \leq k \leq N} \frac{\det \widetilde M^\cR_{\eta-\varepsilon_k}}{\sigma_{\varepsilon_k} \det \widetilde M^\cR_{\eta-\varepsilon_k}},
 \end{equation}
where, as in Section \ref{SD}, $\det \widetilde M^\cR_{\eta-\varepsilon_k}$ is the determinant of the Shapovalov form evaluated on an $\cR$-basis for the $\eta-\varepsilon_k$ weight space of $\widetilde M^\cR$. 
\end{Lemma}

\begin{Comment}
In order for Lemma \ref{equal-dets} to hold as stated, for each $1 \leq k \leq N$, one must calculate the 
$\det \widetilde M^\cR_{\eta-\varepsilon_k}$ in the numerator and denominator with respect 
to the same $\cR$-basis. 
The power of $q$ which appears depends on this choice of $\cR$-bases. 
\end{Comment}

\begin{proof}[Proof of Lemma \ref{equal-dets}]
For each $\gamma \in \hbox{span}_{\ZZ_{\le 0}}(R^+)$ 
fix an $\cR$-basis $B_\gamma$ for $U_q^\cR(\fgl_N)^{<0}_\gamma$.
Consider the following three $\KK$-bases 
for  $\left( ( \widetilde M^\cR \otimes_\cA V)_\eta \right) \otimes_\cR \KK$:
\begin{equation}
\begin{aligned}
&A_\eta  := \{ (b \cdot  v_+) \otimes v_k \ |\  b \in B_{\eta-\varepsilon_k}, 1 \leq k \leq N \}, \\
&C_\eta := \{ b \cdot ( v_+ \otimes v_k) \ |\  b \in B_{\eta-\varepsilon_k}, 1 \leq k \leq N \},
\\
&D_\eta:= \{ b \cdot  v_{\varepsilon_k\tiny+} \ |\  b \in B_{\eta-\varepsilon_k}, 1 \leq k \leq N \}.
\end{aligned}
\end{equation}
Let $\det (\widetilde M^\cR \otimes_\cA V)_B$ denote the determinant of 
$(\cdot, \cdot)_{(\widetilde M^\cR \otimes_\cA V)_\eta}$ calculated on $B$, where 
$B$ is one of $A_\eta, C_\eta$ or $D_\eta$. Let $\det {}^\nu \widetilde M^\cR_{B_{\eta-\nu}}$ 
denote $\det{}^\nu \widetilde M^\cR_\eta$ calculated with respect to the basis 
$B_{\eta-\nu} \cdot v_{\nu+}$.

By the definition of the $\omega$-contravariant form on $\widetilde M^\cR \otimes_\cA V$ (see Section \ref{universal-tensor}),
\begin{equation} \label{det1}
\det (\widetilde M^\cR \otimes V)_{A_\eta} = \prod_{k=1}^N  (\det \widetilde M^\cR_{B_{\eta-\varepsilon_k}})^{\dim V_{\varepsilon_k}} (\det V_{\varepsilon_k})^{\dim \widetilde M^\cR_{\eta-\varepsilon_k}}.
\end{equation}
For $1 \leq k \leq N$, $V_{\varepsilon_k}$ is one dimensional and $\det V_{\varepsilon_k}$ is a power of $q$. Hence, 
up to multiplication by a power of $q$, \eqref{det1} simplifies to
\begin{equation} \label{det2}
\det (\widetilde M^\cR \otimes_\cA V)_{A_\eta}
= \prod_{k=1}^N \det \widetilde M^\cR_{B_{\eta-\varepsilon_k}}.
\end{equation} 

Notice that $U_q^\cA (\fgl_N)^{<0} \cdot v_{\varepsilon_k+}$ is isomorphic to 
${}^{\varepsilon_k} \widetilde M$, and $D_\eta$ is the union of $\cR$-bases for 
each of these submodules. 
For each $1 \leq k \leq N$, and each $\eta \in \mathfrak{h}_\bz^*,$ define an $\cR$ basis of ${}^{\varepsilon_k} \widetilde M_\eta$ by
\begin{equation}
{}^{\varepsilon_k} \widetilde B_\eta:= 
\{ b \cdot v_{\varepsilon_k+} \ |\  b \in B_{\eta-\varepsilon_k} \}.
\end{equation}
Using
$(v_{\varepsilon_k+}, v_{\varepsilon_k+} ) =s_k$, 
\begin{equation} \label{newdet}
\det (\widetilde M^\cR  \otimes V)_{D_\eta}=
 \prod_{k=1}^N s_k^{\dim( {}^{\varepsilon_k} \widetilde M^\cR_\eta)} \det {}^{\varepsilon_k} \widetilde M^\cR_{({}^{\varepsilon_k} \widetilde B_{\eta})} =
\prod_{k=1}^N s_k^{p(\eta-\varepsilon_k)}
\sigma_{\varepsilon_k}(\det \widetilde M^\cR_{\widetilde B_{\eta-\varepsilon_k}}),  
\end{equation}
where the last equality uses Proposition \ref{mu-s}. Here, as in Section \ref{SD}, $\det {}^{\varepsilon_k} \widetilde M^\cR_{({}^{\varepsilon_k} \widetilde B_{\eta})}$ is the Shapovalov determinant calculated with respect to the basis ${}^{\varepsilon_k} \widetilde B_{\eta}$.

The change of basis from $A_\eta$ to $C_\eta$ is unitriangular and 
the change of basis from $C_\eta$ to $D_\eta$ is unitriangular. 
Thus $\det (\widetilde M^\cR \otimes_\cA V)_{A_\eta } =\det (\widetilde M^\cR \otimes_\cA V)_{D_\eta }$, and so the right sides of \eqref{det2} and \eqref{newdet} are equal. The lemma follows from this equality by rearranging.  
 \end{proof}

\begin{Lemma} \label{get_a_hold} Up to multiplication by a power of $q$, 
\begin{equation} s_k = \prod_{1 \leq j < k} \left(  
\frac{z_j^{}z_k^{-1}- q^{2 +2j-2k}z_j^{-1}z_k^{}}
{\sigma_{\varepsilon_j}\left( z_j^{}z_k^{-1}- q^{2 +2j-2k}z_j^{-1}z_k^{} \right)} 
\right).
\end{equation}
\end{Lemma}

\begin{proof}

Fix $1 \leq k \leq N$. Setting $\eta=\varepsilon_k$ in Lemma \ref{equal-dets} 
and applying Theorem \ref{fulldet} we see that, up to multiplication by a power of $q$,
\begin{equation}
\begin{aligned} 
 \prod_{1 \leq x \leq N} s_x^{p(\varepsilon_k-\varepsilon_x)}  &=  \hspace{-0.3cm}
 \prod_{1 \leq x \leq N} \frac{\det \widetilde M^\cR_{\varepsilon_k-\varepsilon_x}}{\sigma_{\varepsilon_x} \det \widetilde M^\cR_{\varepsilon_k-\varepsilon_x}} \\
 &=
  \hspace{-0.3cm}  
\prod_{1 \leq x \leq N} \prod_{\tiny \begin{array}{c} 1 \leq i < j \leq N \\ m >0 \end{array}}  \hspace{-0.4cm}  \left( \hspace{-0.1cm} \frac{c_{\varepsilon_k-\varepsilon_x} \left(z_i^{}z_j^{-1}- q^{2m+2i-2j } z_i^{-1}z_j^{}\right)}{\sigma_{\varepsilon_x}(c_{\varepsilon_k-\varepsilon_x}) \sigma_{\varepsilon_x} \left( z_i^{}z_j^{-1} - q^{2m+2i-2j } z_i^{-1}z_j^{}\hspace{-0.1cm}   \right) } \right)^{ \hspace{-0.15cm} p(\varepsilon_k-\varepsilon_x + m\varepsilon_i-m\varepsilon_j)} \hspace{-1.8cm},
\end{aligned}
\end{equation}
where, for each $ 1 \leq x \leq N$, $c_{\varepsilon_k-\varepsilon_x}$ is a unit in 
$\bq(q)[z_1^{\pm1}, \ldots, z_N^{\pm1}]$.
The value $p(\varepsilon_k-\varepsilon_x+m\varepsilon_i-m\varepsilon_j)$ is $0$ unless $m=1$ and $x \leq i < j \leq k$. 
If $i>x$, then $\sigma_{\varepsilon_x}$ acts as the identity on $z_iz_j^{-1} - q^{2+2i-2j}z_i^{-1}z_j$, so the corresponding factors in the numerator and denominator cancel. Hence we need only consider factors on the right hand side where $m=1$, $i=x$, and $x< j \leq k$. If $x >k$ then $\varepsilon_k-\varepsilon_x \not\in Q^-$, and hence $p(\varepsilon_k-\varepsilon_x)=0$, so on the left hand since we only need to consider those factors where $1 \leq x \leq k$. Up to multiplication by a power of $q$, the expression reduces to
\begin{equation} \label{confusing-product}
\begin{aligned} 
 \hspace{-0.3cm}
 \prod_{1 \leq x \leq k} s_x^{p(\varepsilon_k-\varepsilon_x)}  &=
  \hspace{-0.3cm}
 \prod_{1 \leq x < k} 
  \hspace{-0.1cm}
 \displaystyle \left( \frac{c_{\varepsilon_k-\varepsilon_x}}{\sigma_{\varepsilon_x}(c_{\varepsilon_k- \varepsilon_x})} \right)^{p(\varepsilon_k-\varepsilon_j)} 
 \hspace{-0.3cm}
 \prod_{x< j \leq k}
  \hspace{-0.1cm}
  \left( 
 \frac{z_x^{}z_j^{-1}- q^{2 +2x-2j}z_x^{-1}z_j^{}}
 {\sigma_{\varepsilon_x} \left( z_x^{}z_j^{-1}- q^{2 +2x-2j}z_x^{-1}z_j^{} \right)} 
 \right)^{p(\varepsilon_k-\varepsilon_j)} 
 \hspace{-0.7cm}  \\ 
 & =  \hspace{-0.2cm}  \prod_{1 < j \leq k}  \left(  \prod_{1 \leq x < j}\frac{z_x^{}z_j^{-1}- q^{2 +2x-2j}z_x^{-1}z_j^{}}{{\sigma_{\varepsilon_x}} \left( z_x^{}z_j^{-1}- q^{2 +2x-2j}z_x^{-1}z_j^{} \right)} \right)^{p(\varepsilon_k-\varepsilon_j)}  \hspace{-0.1cm}.
\end{aligned}
\end{equation}
The last two expressions are equal because they are each a product over 
pairs $(x,j)$ with $1 \leq x<j \leq k$, and the factors of 
$ \displaystyle \frac{c_{\varepsilon_k-\varepsilon_x}}
{\sigma_{\varepsilon_x}(c_{\varepsilon_k- \varepsilon_x})}$ have been dropped because they are powers of $q$. Using the fact that $s_1=1$ and making the change of variables 
$j \rightarrow x$ and $x \rightarrow j$ on the right side, \eqref{confusing-product} becomes
\begin{equation}
\begin{aligned} 
 \prod_{1 < x \leq k} s_x^{p(\varepsilon_k-\varepsilon_x)}  
& =  \hspace{-0.2cm}  \prod_{1 < x \leq k}  \left(  
\prod_{1 \leq j < x}\frac{z_j^{}z_x^{-1}- q^{2 +2j-2x}z_j^{-1}z_x^{}}
{\sigma_{\varepsilon_j} \left( z_j^{}z_x^{-1}- q^{2 +2j-2x}z_j^{-1}z_x^{} \right)} 
\right)^{p(\varepsilon_k-\varepsilon_x)}  \hspace{-0.1cm}.
\end{aligned}
\end{equation}
For $k \geq 2$, the lemma now follows by induction. For $k=1$ the result simply says that $s_1=1$, which we already know.
\end{proof}

\begin{figure}
\setlength{\unitlength}{0.4cm}
\begin{center}
\begin{picture}(26,17)

\put(3.65,9.8){$g_1$}
\put(4.65,10.8){$g_2$}
\put(7.65,9.8){$g_3$}
\put(8.65,10.8){$g_4$}
\put(9.65,11.8){$g_5$}
\put(12.65,10.8){$g_6$}
\put(13.65,11.8){$g_7$}
\put(14.65,12.8){$g_8$}
\put(15.65,13.8){$g_9$}
\put(21.45,9.8){$g_{10}$}
\put(22.45,10.8){$g_{11}$}

\put(3,10){\begin{picture}(1,2)
\put(0,2){\line(1,-1){1}}
\put(0,2){\line(-1,-1){1}}
\put(-0.4,0.75){$a_1$}
\end{picture}}

\put(7,10){\begin{picture}(1,2)
\put(0,2){\line(1,-1){1}}
\put(0,2){\line(-1,-1){1}}
\put(-0.4,0.75){$a_3$}
\end{picture}}

\put(12,11){\begin{picture}(1,2)
\put(0,2){\line(1,-1){1}}
\put(0,2){\line(-1,-1){1}}
\put(-0.3,0.75){$b$}
\end{picture}}

\thinlines

\put(12,1){\line(1,1){11}}
\put(11,2){\line(1,1){9}}
\put(10,3){\line(1,1){9}}
\put(9,4){\line(1,1){9}}
\put(8,5){\line(1,1){9}}
\put(7,6){\line(1,1){5}}
\put(6,7){\line(1,1){5}}
\put(5,8){\line(1,1){2}}
\put(4,9){\line(1,1){2}}

\put(14,1){\line(-1,1){10}}
\put(15,2){\line(-1,1){8}}
\put(16,3){\line(-1,1){8}}
\put(17,4){\line(-1,1){8}}
\put(18,5){\line(-1,1){6}}
\put(19,6){\line(-1,1){6}}
\put(20,7){\line(-1,1){6}}
\put(21,8){\line(-1,1){6}}
\put(22,9){\line(-1,1){1}}
\put(23,10){\line(-1,1){1}}
\thicklines

\put(13,0){\line(1,1){13}}
\put(13,0){\line(-1,1){13}}

\put(24,11){\line(-1,1){1}}
\put(23,12){\line(-1,-1){2}}
\put(21,10){\line(-1,1){5}}
\put(16,15){\line(-1,-1){4}}
\put(12,11){\line(-1,1){2}}
\put(10,13){\line(-1,-1){3}}
\put(7,10){\line(-1,1){2}}
\put(5,12){\line(-1,-1){2}}

\put(13,0.05){\line(1,1){13}}
\put(13,0.05){\line(-1,1){13}}

\put(24,11.05){\line(-1,1){1}}
\put(23,12.05){\line(-1,-1){2}}
\put(21,10.05){\line(-1,1){5}}
\put(16,15.05){\line(-1,-1){4}}
\put(12,11.05){\line(-1,1){2}}
\put(10,13.05){\line(-1,-1){3}}
\put(7,10.05){\line(-1,1){2}}
\put(5,12.05){\line(-1,-1){2}}

\put(13,0.1){\line(1,1){13}}
\put(13,0.1){\line(-1,1){13}}

\put(24,11.1){\line(-1,1){1}}
\put(23,12.1){\line(-1,-1){2}}
\put(21,10.1){\line(-1,1){5}}
\put(16,15.1){\line(-1,-1){4}}
\put(12,11.1){\line(-1,1){2}}
\put(10,13.1){\line(-1,-1){3}}
\put(7,10.1){\line(-1,1){2}}
\put(5,12.1){\line(-1,-1){2}}

\put(10.5,10.75){{\tiny $[2]$}}
\put(9.5,9.75){{\tiny $[3]$}}
\put(8.5,8.75){{\tiny $[4]$}}
\put(7.5,7.75){{\tiny $[7]$}}
\put(6.5,6.75){{\tiny $[8]$}}

\put(12.6,-1){ $\tiny{{}_{0}}$}
\put(13.1,-1){ $\tiny{{}_{-1}}$}
\put(14.1,-1){ $\tiny{{}_{-2}}$}
\put(15.1,-1){ $\tiny{{}_{-3}}$}
\put(16.1,-1){ $\tiny{{}_{-4}}$}
\put(17.1,-1){ $\tiny{{}_{-5}}$}
\put(18.1,-1){ $\tiny{{}_{-6}}$}
\put(19.1,-1){ $\tiny{{}_{-7}}$}
\put(20.1,-1){ $\tiny{{}_{-8}}$}
\put(21.1,-1){ $\tiny{{}_{-9}}$}

\put(11.6,-1){ $\tiny{{}_{1}}$}
\put(10.6,-1){ $\tiny{{}_{2}}$}
\put(9.6,-1){ $\tiny{{}_{3}}$}
\put(8.6,-1){ $\tiny{{}_{4}}$}
\put(7.6,-1){ $\tiny{{}_{5}}$}
\put(6.6,-1){ $\tiny{{}_{6}}$}
\put(5.6,-1){ $\tiny{{}_{7}}$}
\put(4.6,-1){ $\tiny{{}_{8}}$}
\put(3.6,-1){ $\tiny{{}_{9}}$}

\end{picture}

$\mbox{}$

\end{center}

\caption{
The partition enclosed by the thick lines is $\lambda=(10, 10, 8, 8, 8, 6, 6, 6, 6,$ $1, 1)$. 
If $k=6$ then $A(\lambda, <6)= \{ a_1, a_3 \}$,  $R(\lambda, <6)= \{g_2, g_5\}$, and  
$$\hspace{-0.7in} \ev_\lambda(s_6)= \frac{[2]}{[3]}\frac{[3]}{[4]}\frac{[4]}{[5]} \frac{[7]}{[8]} \frac{[8]}{[9]} = \frac{[2][7]}{[5][9]}= \frac{[c(g_5)-c(b)][c(g_2)-c(b)]}{[c(a_3)-c(b)][c(a_1)-c(b)]}.
$$
The factors in the numerator of the first expression are displayed.  These are
the $q$-integers corresponding to the hook lengths of the boxes in the same column 
as the addable box $b$ in row 6.  
 \label{ar_fig}}
\end{figure}

\begin{Proposition} \label{the_vals1} 
Let $\lambda$ be a partition. Let $A(\lambda,<k)$ (resp. $R(\lambda, <k)$) be the set of boxes which can be added to (resp. removed from)  $\lambda$ on rows $\lambda_j$ with $j<k$ such that the result is still a partition. Let $b=(\lambda+\varepsilon_k) / \lambda$ and let 
$c(\cdot)$ be as in Figure \ref{partition_bij}. Then, up to multiplication by a power of $q$,
\begin{equation} \label{canceled_eq}
\ev_\lambda(s_k)= 
\begin{cases}
\displaystyle{  \frac{
{\prod_{r \in R(\lambda,<k)} [c(r)-c(b)] }} { {\prod_{a \in A(\lambda,<k) } [c(a)-c(b)]}} 
},
&\text{ if } \lambda+ \varepsilon_k \text{ is a partition}, \\
0, &\text{ if } \lambda+ \varepsilon_k \text{ is not a partition}. 
\end{cases}
\end{equation}

\end{Proposition}

\begin{proof} For $1 \leq j \leq N$, let $g_j$ be the last box in row $j$ of $\lambda$. By Lemma \ref{get_a_hold}, up to multiplication by a power of $q$,
\begin{equation} 
\ev_\lambda^{}(s_k)= \ev_\lambda^{} \left(  \prod_{1 \leq j < k} \frac{z_jz_k^{-1}- q^{2 +2j-2k}z_j^{-1}z_k}{\sigma_{\varepsilon_j} ( z_jz_k^{-1}- q^{2 +2j-2k}z_j^{-1}z_k  )} \right)^{}=
 \prod_{1 \leq j < k}  \frac{
 [c(g_j)-c(b)]}{[c(g_j)-c(b)+1]},
\end{equation}
where the last equality is a simple calculation from definitions. The denominator on the right side is never zero, and the numerator is zero exactly when $\lambda_k=\lambda_{k-1}$, so that $\lambda+\varepsilon_k$ is no longer a partition. If $\lambda_j =\lambda_{j+1}$ for any $j<k$, then there is cancellation, giving \eqref{canceled_eq}. See Figure \ref{ar_fig}.
\end{proof}

\begin{Proposition} \label{the_vals2} 
Let $N_{\bar j}^l(\mu / \lambda)$ be as in Section \ref{MM_section}.
For any partition $\lambda$, 
\begin{equation}
\begin{cases}
val_{\phi_{2\ell}} \ev_\lambda(s_k)= N_{\bar i}^l(\mu / \lambda), 
& \hspace{-0.35cm} \text{ if } \mu =\lambda+\varepsilon_k \text{ is a partition, and } \mu / \lambda \text{ is an } {\bar i} \text{ colored box},\\
\ev_\lambda(s_k)=0, &  \hspace{-0.35cm}  \text{ otherwise}.
\end{cases}
\end{equation}
\end{Proposition}

\begin{proof} By Proposition \ref{the_vals1}, $\ev_\lambda(s_k)=0$ if $\lambda+\varepsilon_k$ is not a partition. If $\lambda+\varepsilon_k$ is a partition then
\begin{equation}
\begin{aligned}
\{ b \in A(\lambda, <k) : \bar c(b) & = \bar c(\mu /\lambda) \}
= \{ b \in A_{\bar i}(\lambda) \ |\  b \text{ is to the left of } \mu/\lambda \}, \quad \text{and}  
\\
\{ b \in R(\lambda, <k) : \bar c(b) & = \bar c(\mu /\lambda) \}
= \{ b \in R_{\bar i}(\lambda)  \ |\  b \text{ is to the left of } \mu/\lambda \},
\end{aligned}
\end{equation}
where the notation is as in Section \ref{MM_section}.  Since 
\begin{equation}
[x] = \frac{q^x-q^{-x}}{q-q^{-1}}= q^{-x} (q-q^{-1})^{-1} \prod_{d | 2x} \phi_d, 
\end{equation}
$[x]$ is divisible by $\phi_{2\ell}$ if and only if $x$ is divisible by $\ell$, and $[x]$ is never divisible by $\phi_{2\ell}^2$. The result now follows from Proposition \ref{the_vals1}.
\end{proof}

\begin{proof}[Proof of Theorem \ref{main}]
Fix $\lambda$ and $1 \leq k \leq m_\lambda$. From definitions, $(\ev_\lambda\otimes 1)v_{\varepsilon_{k_j} \hspace{-0.08cm}+}=v_{\mu^{(j)}}$. Thus, using \eqref{ev-tensor},
\begin{align}
r_{j}(\lambda) =(v_{\mu^{(j)}}, v_{\mu^{(j)}})  = ((\ev_\lambda \otimes 1)v_{\varepsilon_{k_j} \hspace{-0.08cm}+}, (\ev_\lambda \otimes 1)v_{\varepsilon_{k_j} \hspace{-0.08cm}+}) = \ev_\lambda (v_{\varepsilon_{k_j} \hspace{-0.08cm}+}, v_{\varepsilon_{k_j} \hspace{-0.08cm}+}) = \ev_\lambda(s_{k_j}).
\end{align}
The result now follows from Proposition \ref{the_vals2}.
\end{proof}

\def\cprime{$'$} \def\cprime{$'$} \def\cprime{$'$} \def\cprime{$'$}
  \def\cprime{$'$} \def\cprime{$'$}

\end{document}